\documentclass{amsart}
\usepackage{epsfig}
\usepackage{url}        

\newtheorem{theorem}{Theorem}[section]
\newtheorem{lemma}[theorem]{Lemma}
\newtheorem{corollary}[theorem]{Corollary}

\theoremstyle{definition}

\theoremstyle{remark}

\numberwithin{equation}{section}

\begin{document}

\title[Series involving central binomial coefficients]{Further classes of 
series involving central binomial coefficients}

\author{Karl Dilcher}
\address{Department of Mathematics and Statistics\\
 Dalhousie University\\
         Halifax, Nova Scotia, B3H 4R2, Canada}
\email{dilcher@mathstat.dal.ca}

\author{Christophe Vignat}
\address{CentraleSup\'elec, Universit\'e Paris-Saclay, Gif-sur-Yvette, France and Department of
Mathematics, Tulane University, New Orleans, LA 70118, USA}
\email{cvignat@tulane.edu}
\thanks{The first author was supported in part by the Natural Sciences and 
Engineering Research Council of Canada}

\setcounter{equation}{0}

\begin{abstract}
Departing from a class of infinite series with central binomial coefficients in
the numerator and depending on a positive integer parameter, we first extend
known identities to all complex parameters. Then we use various methods,
including exponential Bell polynomials and integral representations, to further
extend these results. Throughout the paper, we make extensive use of the gamma 
and polygamma functions and their properties.
\end{abstract}

\maketitle

\section{Introduction}\label{sec:1}

Infinite series involving central binomial coefficients have been studied for
a long time and continue to be of great interest. A particularly interesting 
and informative paper on this subject was published by D.~H.~Lehmer \cite{Le}
in 1985, which contains a number of methods for obtaining such series, along
with numerous examples. 

We begin here with the pair of series
\begin{equation}\label{1.1}
\sum_{k=0}^\infty\frac{\binom{2k}{k}}{4^k(2k+1)}=\frac{\pi}{2},\qquad
\sum_{k=0}^\infty\frac{\binom{2k}{k}}{4^k(2k+2)}=1.
\end{equation}
Both identities can be obtained from the power series
\begin{equation}\label{1.1a}
\sum_{k=0}^\infty\binom{2k}{k}x^k = \frac{1}{\sqrt{1-4x}},
\end{equation}
which is a special case of the binomial theorem. For the first identity in
\eqref{1.1} we replace $x$ by $x^2/4$ in \eqref{1.1a} and integrate; this 
gives a well-known series for the arcsine function which we evaluate at $x=1$.
Similarly, the second identity in \eqref{1.1} follows from integrating
\eqref{1.1a} as is. For details, see \cite[p.~449--451]{Le}.

Our recent paper \cite{DV} contains a wider study of power series and
numerical series including central binomial coefficients, which arise from 
integration methods applied to the arcsine function and some of its powers. 
As special cases we obtained the following classes of series:

For all integers $\ell\geq 0$ we have
\begin{align}
\sum_{k=0}^\infty\frac{\binom{2k}{k}}{4^k(2k+2\ell+1)}
&=\binom{2\ell}{\ell}\frac{\pi}{2^{2\ell+1}},\label{1.2} \\
\sum_{k=0}^\infty\frac{\binom{2k}{k}}{4^k(2k+2\ell+2)}
&= \frac{2^{2\ell}}{(2\ell+1)\binom{2\ell}{\ell}}.\label{1.3}
\end{align}
These clearly generalize the two identities in \eqref{1.1}. Next, we shall see
that the two identities \eqref{1.2} and \eqref{1.3} can be combined if we use
the gamma function instead of binomial coefficients. A proof will be given in
Section~\ref{sec:2}.

\begin{theorem}\label{thm:1.1}
For all integers $n\geq 0$ we have
\begin{equation}\label{1.4}
\sum_{k=0}^\infty\frac{\binom{2k}{k}}{4^k(2k+1+n)}
=\frac{\sqrt{\pi}}{2}\cdot\frac{\Gamma(\frac{n+1}{2})}{\Gamma(\frac{n+2}{2})}.
\end{equation}
\end{theorem}

The first goal of this paper is to extend \eqref{1.4} to almost all complex
parameters. This will be done in Section~\ref{sec:2}, along with some special
cases. For instance, in Corollary~\ref{cor:2.2} we obtain the surprising
infinite class of identities
\[
\sum_{k=0}^\infty\frac{\binom{2k}{k}}{4^k(2k+1-2m)}=0,\quad m=1, 2, 3,\ldots.
\]
The cases not covered by the extension of Theorem~\ref{thm:1.1} are treated
separately in Section~\ref{sec:3}. We then use differentiation to obtain
further identities in Section~\ref{sec:4} that have squares and cubes of the 
linear factors in the denominator. In Section~\ref{sec:5} we use
exponential Bell polynomials to extend the previous results to arbitrary
positive integer powers of the linear factor in the denominator. Some integral 
representations are then used in Section~\ref{sec:6} for an alternative
approach to some of our results, and we conclude this paper with a few further
remarks and identities in Section~\ref{sec:7}.

\section{Extending the identity \eqref{1.4}}\label{sec:2}

We begin with a proof of Theorem~\ref{thm:1.1}.

\begin{proof}[Proof of Theorem~\ref{thm:1.1}]
Our main tool for dealing with the gamma function is Legendre's duplication
formula
\begin{equation}\label{2.1}
\Gamma(2z) = \pi^{-1/2} 2^{2z-1}\Gamma(z)\Gamma(z+\tfrac{1}{2}),
\end{equation}
valid for $2z\neq 0, -1, -2,\ldots$; see, e.g., \cite[eq.~5.5.5]{DLMF}.
First, for \eqref{1.2} and \eqref{1.4} to agree, we require
\[
\frac{\Gamma(2\ell+1)}{\Gamma(\ell+1)^2}\cdot\frac{\pi}{2^{2\ell+1}}
=\frac{\sqrt{\pi}}{2}\cdot\frac{\Gamma(\ell+\tfrac{1}{2})}{\Gamma(\ell+1)}.
\]
But this is easily seen to be equivalent to \eqref{2.1} with
$z=\ell+\tfrac{1}{2}$. Second, for \eqref{1.3} and \eqref{1.4} to agree, we
similarly require
\[
\frac{2^{2\ell}\Gamma(\ell+1)^2}{(2\ell+1)\Gamma(2\ell+1)}
=\frac{\sqrt{\pi}}{2}\cdot\frac{\Gamma(\ell+1)}{\Gamma(\ell+\tfrac{3}{2})}.
\]
This is seen to be equivalent to \eqref{2.1} with $z=\ell+1$ if we also note
that $(2\ell+1)\Gamma(2\ell+1)=\Gamma(2\ell+2)$.
\end{proof}

We are now ready to state and prove the main result of this section.

\begin{theorem}\label{thm:2.1}
For any $z\in{\mathbb C}\setminus\{-1,-3,-5,\ldots\}$ we have
\begin{equation}\label{2.2}
\sum_{k=0}^\infty\frac{\binom{2k}{k}}{4^k(2k+1+z)}
=\frac{\sqrt{\pi}}{2}\cdot\frac{\Gamma(\frac{z+1}{2})}{\Gamma(\frac{z+2}{2})}.
\end{equation}
\end{theorem}

\begin{proof}
We first assume that $z$ is real and not an odd negative integer. Then, using
the integral 
\[
\int_0^1 t^{2k+z}dt = \frac{1}{2k+1+z}
\]
and the series \eqref{1.1a}, we get
\begin{equation}\label{2.3}
\sum_{k=0}^\infty\frac{\binom{2k}{k}}{4^k(2k+1+z)}
=\int_0^1\sum_{k=0}^\infty\frac{\binom{2k}{k}}{4^k}t^{2k+z}dt 
=\int_0^1\frac{t^z}{\sqrt{1-t^2}}dt,
\end{equation}
where the change in the order of integration and summation can be justified.
Next, with the change of variables $w=t^2$ and using Euler's beta integral
\begin{equation}\label{2.3a}
B(a,b):=\int_0^1 w^{a-1}(1-w)^{b-1}dw = \frac{\Gamma(a)\Gamma(b)}{\Gamma(a+b)}
\end{equation}
(see, e.g., \cite[eq.~5.12.1]{DLMF}), we get
\begin{equation}\label{2.4}
\int_0^1\frac{t^z}{\sqrt{1-t^2}}dt
=\frac{1}{2}\int_0^1 w^{\frac{z}{2}-\frac{1}{2}}(1-w)^{-\frac{1}{2}}dw
=\frac{\Gamma(\frac{z}{2}-\frac{1}{2})\Gamma(\frac{1}{2})}{2\Gamma(\frac{1}{2}+1)}.
\end{equation}
Since $\Gamma(\frac{1}{2})=\sqrt{\pi}$, \eqref{2.4} together with \eqref{2.3}
prove \eqref{2.2} for real $z$. Finally, by the well-known asymptotics for 
the central binomial coefficients we have
\begin{equation}\label{2.5}
\frac{4^k}{\sqrt{\pi k}}\cdot\frac{7}{8}
\leq\frac{4^k}{\sqrt{\pi k}}\bigl(1-\frac{1}{8k}\bigr) < \binom{2k}{k}
< \frac{4^k}{\sqrt{\pi k}}\quad (k\geq 1),
\end{equation}
and by properties of the gamma function, both sides of \eqref{2.2} are 
analytic in $z\in{\mathbb C}\setminus\{-1,-3,-5,\ldots\}$. By the identity 
theorem, they are therefore identical in this region, which proves the theorem.
\end{proof}

In addition to \eqref{1.2} and \eqref{1.3}, which we recover via \eqref{1.4}, 
we obtain the following special cases from Theorem~\ref{thm:2.1}.

\begin{corollary}\label{cor:2.2}
\begin{align}
\sum_{k=0}^\infty\frac{\binom{2k}{k}}{4^k(2k+1-2m)} &=0,\quad 
m=1,2,3,\ldots,\label{2.6}\\
\sum_{k=0}^\infty\frac{\binom{2k}{k}}{4^k(2k+\frac{1}{2})} 
&=\frac{\Gamma(\tfrac{1}{4})^2}{2\sqrt{2\pi}},\label{2.7}\\
\sum_{k=0}^\infty\frac{\binom{2k}{k}}{4^k(2k-\frac{1}{2})} 
&=-\frac{(2\pi)^{3/2}}{\Gamma(\tfrac{1}{4})^2}.\label{2.8}
\end{align}
\end{corollary}

\begin{proof}
The denominator of the right-hand side of \eqref{2.2} has a pole at $z=-2m$,
while the gamma function in the numerator is finite at $z=-2m$. This proves 
\eqref{2.6}. The remaining two identities follow from some known evaluations
and elementary properties of the gamma function.
\end{proof}

More generally, using the evaluations that go into \eqref{2.7} and 
\eqref{2.8}, along with the fundamental relation $\Gamma(z+1)=z\Gamma(z)$, 
one can easily derive general identities for 
\[
\sum_{k=0}^\infty\frac{\binom{2k}{k}}{4^k(2k+m+\frac{1}{2})},
\]
where $m$ is an arbitrary integer.

We now use the fact that \eqref{2.2} holds for all non-real $z$.

\begin{corollary}\label{cor:2.3}
For all real $y\neq 0$ we have
\begin{equation}\label{2.9}
\left|\sum_{k=0}^\infty\frac{\binom{2k}{k}}{4^k(2k+1+iy)}\right|^2
=\frac{\pi}{2y}\cdot\tanh(\tfrac{\pi y}{2}).
\end{equation}
\end{corollary}

\begin{proof}
We use \eqref{2.2} with $z=iy$. By identity (5.4.4) in \cite{DLMF}, we have
\[
\left|\Gamma(\tfrac{iy}{2}+\tfrac{1}{2})\right|^2
=\frac{\pi}{\cosh(\tfrac{\pi y}{2})},
\]
and by (5.4.3) in \cite{DLMF},
\[
\left|\Gamma(\tfrac{iy}{2}+1)\right|^2
=\left|\tfrac{iy}{2}\Gamma(\tfrac{iy}{2})\right|^2
=\frac{y^2}{4}\cdot\frac{2\pi}{y\sinh(\tfrac{\pi y}{2})}.
\]
The result now follows by combining these last two identities with 
\eqref{2.2}.
\end{proof}

Although the right-hand side of \eqref{2.9} is not defined at $y=0$, the 
limit as $y\to 0$ exists and is $\pi^2/4$. We therefore recover the
identity
\[
\sum_{k=0}^\infty\frac{\binom{2k}{k}}{4^k(2k+1)} =\frac{\pi}{2},
\]
which is the case $\ell=0$ in \eqref{1.2}. This identity also follows from 
\eqref{2.2} with $z=0$ and using the fact that $\Gamma(\frac{1}{2})=\sqrt{\pi}$.

\section{The missing values in Theorem~\ref{thm:2.1}}\label{sec:3}

Considering the identity \eqref{2.2}, we see that both sides are indeed not
defined at odd negative integers. On the other hand, if we set $z=-2m-1$,
$m=0, 1, 2,\ldots$ and disregard the term $k=m$ in the summation, then the 
singularity on the left is removed. Using the relation \eqref{2.5}, we 
see that the series converges, just as it does for all other values of $z$.
Also, in the case $m=0$, or equivalently $z=-1$, we have the well-known 
identity
\begin{equation}\label{3.1}
\sum_{k=1}^\infty\frac{\binom{2k}{k}}{4^k(2k)} =\log{2},
\end{equation}
which can be found, for instance, in \cite[p.~450]{Le}.

These considerations lead to the following result and its proof. We require the
harmonic numbers, which are defined by
\begin{equation}\label{3.2}
H_0:=1\quad\hbox{and}\quad H_n:=\sum_{j=1}^n\frac{1}{j},\quad n=1, 2, 3,\ldots.
\end{equation}

\begin{theorem}\label{thm:3.1}
For all integers $m\geq 0$ we have
\begin{equation}\label{3.3}
\sum_{\substack{k=0\\k\neq m}}^\infty\frac{\binom{2k}{k}}{4^k(2k-2m)}
=\binom{2m}{m}\cdot\frac{\log{2}-H_{2m}+H_m}{4^m}.
\end{equation}
\end{theorem}

Obviously, with $m=0$ we get \eqref{3.1}. The next few cases are
\begin{align*}
\sum_{\substack{k=0\\k\neq 1}}^\infty\frac{\binom{2k}{k}}{4^k(2k-2)}
&=\frac{\log{2}}{2}-\frac{1}{4},\\
\sum_{\substack{k=0\\k\neq 2}}^\infty\frac{\binom{2k}{k}}{4^k(2k-4)}
&=\frac{3\log{2}}{8}-\frac{7}{32},\\
\sum_{\substack{k=0\\k\neq 3}}^\infty\frac{\binom{2k}{k}}{4^k(2k-6)}
&=\frac{5\log{2}}{16}-\frac{37}{192}.
\end{align*}

In the proof of Theorem~\ref{thm:3.1} and later in this paper we will make use
of the digamma function, defined by
\begin{equation}\label{3.4}
\psi(z)=\frac{d}{dz}\log{\Gamma(z)} = \frac{\Gamma'(z)}{\Gamma(z)}.
\end{equation}
According to the identities (5.4.14) and (5.4.15) in \cite{DLMF}, the digamma
function has the special values
\begin{equation}\label{3.5}
\psi(m+1) = H_m-\gamma,\qquad\psi(m+\tfrac{1}{2}) = 2H_{2m}-H_m-2\log{2}-\gamma,
\end{equation}
where $\gamma$ is Euler's constant. The function $\psi(z)$ also satisfies the
reflection formula
\begin{equation}\label{3.6}
\psi(z) = \psi(1-z)-\pi\cot(\pi z),\quad z\not\in{\mathbb Z}
\end{equation}
(see, e.g., \cite[eq.~5.5.4]{DLMF}), and with $z=-m+\frac{1}{2}$ the cotangent
term in \eqref{3.6} disappears, and we find with the second identity in 
\eqref{3.5} that 
\begin{equation}\label{3.7}
\psi(-m+\tfrac{1}{2}) = 2H_{2m}-H_m-2\log{2}-\gamma.
\end{equation}
We are now ready to prove the theorem.

\begin{proof}[Proof of Theorem~\ref{thm:3.1}]
We denote the series on the left of \eqref{3.3} by $S(m)$ and note that by
\eqref{2.2} we have 
\begin{align*}
S(m)&=\lim_{z\to -2m-1}\left(
\frac{\sqrt{\pi}}{2}\cdot\frac{\Gamma(\frac{z+1}{2})}{\Gamma(\frac{z+2}{2})}
-\frac{\binom{2m}{m}}{4^m(2m+1+z)}\right) \\
&=\lim_{z\to -2m-1}\frac{1}{2m+1+z}\left(\sqrt{\pi}
\cdot\frac{\Gamma(\frac{z+1}{2})(m+\frac{z+1}{2})}{\Gamma(\frac{z+2}{2})}
-4^{-m}\binom{2m}{m}\right).
\end{align*}
Using L'Hospital's rule, this gives
\[
S(m)=\sqrt{\pi}\cdot\frac{d}{dz}\left.
\frac{\Gamma(\frac{z+1}{2})(m+\frac{z+1}{2})}{\Gamma(\frac{z+2}{2})}
\right|_{z=-2m-1}.
\]
To deal with the removable singularity in the numerator of the last fraction,
we iterate the basic identity $w\Gamma(w)=\Gamma(w+1)$, obtaining
\begin{align*}
\Gamma(\tfrac{z+1}{2})(m+\tfrac{z+1}{2}) &=\frac{\Gamma(m+1+\frac{z+1}{2})}
{(\frac{z+1}{2})(\frac{z+1}{2}+1)\cdots(\frac{z+1}{2}+m-1)} \\
&=\frac{\Gamma(m+1+\frac{z+1}{2})\cdot 2^m}{(z+1)(z+3)\cdots(z+2m-1)},
\end{align*}
so that
\begin{equation}\label{3.8}
S(m)=2^m\sqrt{\pi}\cdot\frac{d}{dz}\left.
\frac{\Gamma(m+1+\frac{z+1}{2})}{\Gamma(\frac{z+2}{2})(z+1)(z+3)\cdots(z+2m-1)}
\right|_{z=-2m-1}.
\end{equation}
For ease of notation, we set
\begin{align*}
G(z)&:=\Gamma(m+1+\tfrac{z+1}{2}),\quad H(z):=\Gamma(\tfrac{z+2}{2}),\\
P(z)&:=(z+1)(z+3)\cdots(z+2m-1).
\end{align*}
Then we have
\begin{equation}\label{3.9}
G(-2m-1) = \Gamma(1) = 1,
\end{equation}
and with \eqref{3.4} and the first identity of \eqref{3.5} we get
\begin{equation}\label{3.10}
G'(-2m-1) =\frac{\Gamma(1)\psi(1)}{2} = -\frac{\gamma}{2}.
\end{equation}
Next, we have
\begin{equation}\label{3.11}
H(-2m-1) = \Gamma(-m+\tfrac{1}{2}) = \frac{(-4)^m m!\sqrt{\pi}}{(2m)!},
\end{equation}
where the second equation can be found, for instance, in 
\cite[eq.~8.339.3]{GR}. Further, with \eqref{3.11} and \eqref{3.7} we get
\begin{align}
H'(-2m-1) 
&= \frac{1}{2}\Gamma(-m+\tfrac{1}{2})\psi(-m+\tfrac{1}{2})\label{3.12} \\
&= \frac{(-4)^m m!\sqrt{\pi}}{2(2m)!}\left(2H_{2m}-H_m-2\log{2}-\gamma\right).
\nonumber
\end{align}
Finally, it is easy to see that
\begin{equation}\label{3.13}
P(-2m-1) = (-2)^m m!,
\end{equation}
and since
\[
P'(z) = P(z)\sum_{j=1}^m\frac{1}{z+2j-1},
\]
we have with \eqref{3.13} and the definition of $H_m$,
\begin{equation}\label{3.14}
P'(-2m-1) = (-2)^{m-1}m!H_m.
\end{equation}
By the standard rules of differentiation, the derivative in \eqref{3.8} becomes
\[
\frac{G'(z)}{H(z)P(z)}-\frac{G(z)P'(z)}{H(z)P(z)^2}
-\frac{G(z)H'(z)}{H(z)^2P(z)}.
\]
We now use the identities \eqref{3.9}--\eqref{3.14} to evaluate this last 
expression at $z=-2m-1$. After some straightforward manipulations and
cancellations, we get with \eqref{3.8},
\[
S(m)=2^m\sqrt{\pi}\binom{2m}{m}
\cdot\frac{log{2}-H_{2m}+H_m}{8^m\sqrt{\pi}}.
\]
This is clearly the same as the right-hand side of \eqref{3.3}, and the proof
is complete.
\end{proof}

\section{Squares and cubes in the denominators}\label{sec:4}

We can get more identities by differentiating both sides of \eqref{2.2}.
To do so, we use again \eqref{3.4} in the form $\Gamma'(z)=\Gamma(z)\psi(z)$.

\begin{corollary}\label{cor:4.1}
For any $z\in{\mathbb C}\setminus\{-1,-3,-5,\ldots\}$ we have
\begin{equation}\label{4.1}
\sum_{k=0}^\infty\frac{\binom{2k}{k}}{4^k(2k+1+z)^2}
=\frac{\sqrt{\pi}}{4}\cdot\frac{\Gamma(\frac{z+1}{2})}{\Gamma(\frac{z+2}{2})}
\left(\psi(\tfrac{z+2}{2})-\psi(\tfrac{z+1}{2})\right).
\end{equation}
\end{corollary}

As special cases we get the following explicit evaluations.

\begin{corollary}\label{cor:4.2}
For all integers $m\geq 1$, we have
\begin{equation}\label{4.2}
\sum_{k=0}^\infty\frac{\binom{2k}{k}}{4^k(2k+2m)^2}
=\frac{2^{2m-1}}{m\binom{2m}{m}}\left(H_{2m}-H_m+\tfrac{1}{2m}-\log{2}\right),
\end{equation}
and for $m\geq 0$, 
\begin{equation}\label{4.3}
\sum_{k=0}^\infty\frac{\binom{2k}{k}}{4^k(2k+2m+1)^2}
=\frac{\pi\binom{2m}{m}}{2^{2m+1}}\left(H_m-H_{2m}+\log{2}\right).
\end{equation}
In particular,
\begin{align}
\sum_{k=0}^\infty\frac{\binom{2k}{k}}{4^k(2k+1)^2}
&=\frac{\pi\log{2}}{2},\label{4.4}\\
\sum_{k=0}^\infty\frac{\binom{2k}{k}}{4^k(2k+2)^2}
&=1-\log{2},\label{4.5}\\
\sum_{k=0}^\infty\frac{\binom{2k}{k}}{4^k(2k+3)^2}
&=\frac{\pi}{8}(-1+2\log{2}),\label{4.6}\\
\sum_{k=0}^\infty\frac{\binom{2k}{k}}{4^k(2k+4)^2}
&=\frac{5}{9}-\frac{2}{3}\log{2}.\label{4.7}
\end{align}
\end{corollary}

\begin{proof}
In \eqref{4.1} we set $z=2m-1$, $m\geq 1$. Then with the identity
\begin{equation}\label{4.8}
\Gamma(m+\tfrac{1}{2}) = \frac{(2m)!\sqrt{\pi}}{4^m m!}\quad (m=0,1,2,\ldots).
\end{equation}
(see, e.g., \cite[eq.~8.339.2]{GR}) and the fact that $\Gamma(m)=(m-1)!$, we get
\begin{equation}\label{4.9}
\frac{\Gamma(\frac{z+1}{2})}{\Gamma(\frac{z+2}{2})}
=\frac{\Gamma(m)}{\Gamma(m+\frac{1}{2})}
=\frac{4^m}{m\binom{2m}{m}\sqrt{\pi}},
\end{equation}
and with \eqref{3.5},
\begin{equation}\label{4.10}
\psi(\tfrac{z+2}{2})-\psi(\tfrac{z+1}{2})=\psi(m+\tfrac{1}{2})-\psi(m)
= 2H_{2m}-2H_m+\tfrac{1}{m}-2\log{2}.
\end{equation}
The identities \eqref{4.9} and \eqref{4.10}, combined with \eqref{4.1}, then
give the desired identity \eqref{4.2}.

We obtain \eqref{4.3} analogously, using again \eqref{4.8} and \eqref{3.5}.
Finally, by evaluating the appropriate small values of the harmonic numbers,
we easily obtain \eqref{4.4}--\eqref{4.7} from \eqref{4.2} or \eqref{4.3}.
\end{proof}

In analogy to \eqref{2.6}, we obtain easy special values when $z$ is an even
negative integer in \eqref{4.1}. This time, these values are somewhat more 
interesting than in \eqref{2.6}.

\begin{corollary}\label{cor:4.3}
For all integers $m\geq 1$, we have
\begin{equation}\label{4.11}
\sum_{k=0}^\infty\frac{\binom{2k}{k}}{4^k(2k+1-2m)^2}
=\frac{4^{m-1}\pi}{m\binom{2m}{m}}.
\end{equation}
\end{corollary}

\begin{proof}
We know that both $\Gamma(z)$ and $\psi(z)$ have poles of order 1 at integers
$\leq 0$ and are analytic elsewhere in $\mathbb C$. This means that for 
integers $m\geq 1$ we have by \eqref{4.1},
\begin{equation}\label{4.12}
\sum_{k=0}^\infty\frac{\binom{2k}{k}}{4^k(2k+1-2m)^2}
=\frac{\sqrt{\pi}}{4}\cdot\Gamma(-m+\frac{1}{2})
\lim_{z\to 1-m}\frac{\psi(z)}{\Gamma(z)}.
\end{equation}
Since the residues are known to be 
\[
{\rm Res}(\Gamma,1-m)=\frac{(-1)^{m-1}}{(m-1)!}\quad\hbox{and}\quad
{\rm Res}(\psi, 1-m)=-1,
\]
(see, e.g., \cite[Sect.~5.2(i)]{DLMF}), we have
\begin{equation}\label{4.13}
\lim_{z\to 1-m}\frac{\psi(z)}{\Gamma(z)}=(-1)^m(m-1)!.
\end{equation}
Next, by \eqref{4.8} and the reflection formula 
$\Gamma(z)\Gamma(1-z)=\pi/\sin(\pi z)$, we have
\begin{equation}\label{4.14}
\Gamma(-m+\tfrac{1}{2})=(-1)^m\frac{m!4^m}{(2m)!}\sqrt{\pi}.
\end{equation}
This, together with \eqref{4.13} and \eqref{4.12}, gives the desired identity
\eqref{4.11}.
\end{proof}

We can actually go further than Corollary~\ref{cor:4.1} and differentiate 
both sides of \eqref{4.1}. Then we easily obtain the following.

\begin{corollary}\label{cor:4.4}
For any $z\in{\mathbb C}\setminus\{-1,-3,-5,\ldots\}$ we have
\begin{equation}\label{4.15}
\sum_{k=0}^\infty\frac{\binom{2k}{k}}{4^k(2k+1+z)^3}
=\frac{\sqrt{\pi}}{16}\cdot\frac{\Gamma(\frac{z+1}{2})}{\Gamma(\frac{z+2}{2})}
\left(\left(\psi(\tfrac{z+1}{2})-\psi(\tfrac{z+2}{2})\right)^2
+\psi'(\tfrac{z+1}{2})-\psi'(\tfrac{z+2}{2})\right),
\end{equation}
and for $z=0$ and $z=1$ we get respectively
\begin{align}
\sum_{k=0}^\infty\frac{\binom{2k}{k}}{4^k(2k+1)^3}
&= \frac{\pi}{48}\left(\pi^2+12\log^2{2}\right),\label{4.16}\\
\sum_{k=0}^\infty\frac{\binom{2k}{k}}{4^k(2k+2)^3}
&= 1-\frac{\pi^2}{24}-\log{2}+\frac{\log^2{2}}{2}.\label{4.17}
\end{align}
\end{corollary}

\begin{proof}
The identity \eqref{4.15} follows immediately from differentiating both sides
of \eqref{4.1}, using again the relation $\Gamma'(z)=\Gamma(z)\Psi(z)$. To
obtain \eqref{4.16} and \eqref{4.17}, we use some identities from the proof of
Corollary~\ref{cor:4.2}, in addition to 
\[
\psi'(\tfrac{1}{2})=\frac{\pi^2}{2},\quad
\psi'(1)=\frac{\pi^2}{6},\quad
\psi'(\tfrac{3}{2})=\frac{\pi^2}{2}-4.
\]
These last three evaluations follow from the identities (5.15.3), (5.15.2)
and (5.15.5), respectively, in \cite{DLMF}.
\end{proof}

In analogy to \eqref{2.6} and \eqref{4.11}, and using similar methods as in
the proof of \eqref{4.11}, we obtain the following result.

\begin{corollary}\label{cor:4.5}
For all integers $m\geq 1$, we have
\begin{equation}\label{4.18}
\sum_{k=0}^\infty\frac{\binom{2k}{k}}{4^k(2k+1-2m)^3}
=-\frac{4^{m-1}\pi}{m\binom{2m}{m}}\left(H_{2m}-H_m+\frac{1}{2m}-\log{2}\right).
\end{equation}
\end{corollary}

For the proof of \eqref{4.18} we require the following residue.

\begin{lemma}\label{lem:4.6}
For any $m\in{\mathbb N}$ we have
\begin{equation}\label{4.19}
{\rm Res}(\psi(z)^2-\psi'(z),-m+1) = 2(\gamma-H_{m-1}).
\end{equation}
\end{lemma}

\begin{proof}
Since $\psi(z)$ has poles of order 1 with residue $-1$ at all nonpositive 
integers, we can write for a fixed $m\in{\mathbb N}$,
\begin{equation}\label{4.20}
\psi(z) = -\frac{1}{z+m-1}+g(z),
\end{equation}
with $g(z)$ analytic at $1-m$. This implies
\[
\psi(z)^2-\psi'(z) = -\frac{2g(z)}{z+m-1}+g(z)^2-g'(z),
\]
and consequently,
\begin{equation}\label{4.21}
{\rm Res}(\psi(z)^2-\psi'(z),1-m) = -2g(1-m).
\end{equation}
Next, iterating the functional equation $\psi(z+1)=\frac{1}{z}+\psi(z)$ (see,
e.g., \cite[eq.~(5.5.2)]{DLMF}), we get
\[
\psi(z+m) = \sum_{j=0}^{m-1}\frac{1}{z+j} + \psi(z),
\]
and with \eqref{4.20},
\[
g(z)=\psi(z+m) - \sum_{j=0}^{m-2}\frac{1}{z+j}.
\]
Finally, with $z=1-m$ we get
\[
g(1-m)=\psi(1) - \sum_{j=0}^{m-2}\frac{-1}{m-1-j} = -\gamma+H_{m-1},
\]
where we have used the well-known evaluation $\psi(1)=-\gamma$, which is also a
special case of \eqref{3.5}. The desired identity \eqref{4.19} now follows
from \eqref{4.21}.
\end{proof}

\begin{proof}[Proof of Corollary~\ref{cor:4.5}]
We fix $m\in{\mathbb{N}}$, and in \eqref{4.15} we set $z=-2\mu$, with $\mu$ in
a neighborhood of $m$. Then the right-hand side of \eqref{4.15} becomes
\begin{equation}\label{4.22}
\frac{\sqrt{\pi}}{16}\Gamma(\tfrac{1}{2}-\mu)\left(-2\psi(\tfrac{1}{2}-\mu)
\frac{\psi(1-\mu)}{\Gamma(1-\mu)}
+\frac{\psi(1-\mu)^2-\psi'(1-\mu)}{\Gamma(1-\mu)}\right).
\end{equation}
Now we take the limit as $\mu\to m$. We first use the reflection identity
\[
\psi(1-z) = \psi(z)+\pi\cot(\pi z)
\]
(see, e.g., \cite[eq.~(5.5.4)]{DLMF}) with $z=m+\frac{1}{2}$, to obtain 
\begin{equation}\label{4.23}
\psi(\tfrac{1}{2}-m)=\psi(\tfrac{1}{2}+m) = 2H_{2m}-H_m-2\log{2}-\gamma,
\end{equation}
where we have used \eqref{3.5}. Next, we use again 
${\rm Res}(\Gamma, 1-m)=(-1)^{m-1}/(m-1)!$, to obtain from \eqref{4.19},
\begin{equation}\label{4.24}
\lim_{\mu\to m}\frac{\psi(1-\mu)^2-\psi'(1-\mu)}{\Gamma(1-\mu)}
= (-1)^{m-1}2(m-1)!\left(\gamma-H_{m-1}\right).
\end{equation}
Finally, we substitute \eqref{4.23} and \eqref{4.24}, as well as \eqref{4.13}
and \eqref{4.14}, into \eqref{4.22} after the limit is taken. A straightforward
manipulation then gives the desired identity \eqref{4.18}.
\end{proof}

\section{Further extensions}\label{sec:5}

Considering Corollaries~\ref{cor:4.4} and~\ref{cor:4.5} and their proofs, it is
clear that it would become increasingly complicated to obtain expressions for
infinite series of the type \eqref{4.15} with higher powers of $2k+1+z$ in the
denominator. However, using Bell polynomials, it is in fact
possible to obtain identities involving arbitrary powers of $2k+1+z$ in the
denominator. 

The Bell polynomials belong to the most important objects in enumerative
combinatorics. Given a sequence $x_1, x_2,\ldots$, the polynomials
$B_n(x_1,\ldots,x_n)$, $n\geq 0$, are usually defined by their generating 
function
\begin{equation}\label{5.1}
\exp\left(\sum_{k=1}^\infty x_k\frac{z^k}{k!}\right)
=\sum_{n=0}^\infty B_n(x_1,\ldots,x_n)\frac{z^n}{n!};
\end{equation}
see, e.g., \cite[p.~133ff.]{Co}. These polynomials are also known as complete
exponential Bell polynomials. Among numerous other properties, they have the
explicit expansion
\begin{equation}\label{5.2}
B_n(x_1,\ldots,x_n) = n!\sum\frac{1}{j_1j_2!\cdots j_n!}
\left(\frac{x_1}{1!}\right)^{j_1}\left(\frac{x_2}{2!}\right)^{j_2}\cdots
\left(\frac{x_n}{n!}\right)^{j_n},
\end{equation}
where the sum is taken over all non-negative integers $j_1,\ldots,j_n$ that
satisfy $j_1+2j_2+\cdots+nj_n=n$. The Bell polynomials also satisfy the 
recurrence relation $B_0=1$ and for $n\geq 0$,
\begin{equation}\label{5.3}
B_{n+1}(x_1,\ldots,x_{n+1}) 
= \sum_{i=0}^n\binom{n}{i}B_{n-i}(x_1,\ldots,x_{n-i})x_{i+1}.
\end{equation}
Using \eqref{5.2} or \eqref{5.3}, we can easily compute the first few Bell
polynomials, displayed in Table~1. A more extensive table can be found, e.g.,
in \cite[p.~307]{Co}.

\bigskip
\begin{center}
{\renewcommand{\arraystretch}{1.2}
\begin{tabular}{|r|l|}
\hline
$n$ & $B_n(x_1,\ldots,x_n)$ \\
\hline
0 & $1$ \\
1 & $x_1$ \\
2 & $x_1^2+x_2$ \\
3 & $x_1^3+3x_1x_2+x_3$ \\
4 & $x_1^4+6x_1^2x_2+4x_1x_3+3x_2^2+x_4$ \\
5 & $x_1^5+10x_1^3x_2+15x_1x_2^2+10x_1^2x_3+10x_2x_3+5x_1x_4+x_5$ \\
\hline
\end{tabular}}

\medskip
{\bf Table~1}: Bell polynomials $B_n(x_1,\ldots,x_n)$ for $0\leq n\leq 5$.
\end{center}

\bigskip
We also require the polygamma function, defined for integer orders $m\geq 0$ by
\begin{equation}\label{5.4}
\psi^{(m)}(z) = \frac{d^m}{dz^m}\psi(z)=\frac{d^{m+1}}{dz^{m+1}}\log{\Gamma(z)}.
\end{equation}
We are now ready to state and prove the main result of this section, which 
extends the identities \eqref{2.2}, \eqref{4.1}, and \eqref{4.15}.

\begin{theorem}\label{thm:5.1}
For any $z\in{\mathbb C}\setminus\{-1,-3,-5,\ldots\}$ and any integer 
$p\geq 0$ we have
\begin{equation}\label{5.5}
\sum_{k=0}^\infty\frac{\binom{2k}{k}}{4^k(2k+1+z)^{p+1}}
=\frac{(-1)^p\sqrt{\pi}}{2p!}
\cdot\frac{\Gamma(\frac{z+1}{2})}{\Gamma(\frac{z+2}{2})}
\cdot B_{p}(g^{(1)}(z),\ldots,g^{(p)}(z)),
\end{equation}
where for $k\geq 1$,
\begin{equation}\label{5.6}
g^{(k)}(z)
= 2^{-k}\left(\psi^{(k-1)}(\tfrac{z+1}{2})-\psi^{(k-1)}(\tfrac{z+2}{2})\right).
\end{equation}
\end{theorem}

\begin{proof}
We fix $z\in{\mathbb C}\setminus\{-1,-3,-5,\ldots\}$ and for a complex variable
$w$ we define
\begin{equation}\label{5.7}
f(z+w):=2\sum_{k=0}^\infty\frac{\binom{2k}{k}}{4^k(2k+1+z+w)}
=\sqrt{\pi}\cdot\frac{\Gamma(\frac{z+w+1}{2})}{\Gamma(\frac{z+w+2}{2})},
\end{equation}
which is holomorphic in the domain $D:=\{w\mid w+z\not\in{\mathbb Z}_{<0}\}$.
We recall that the second equation in \eqref{5.7} comes from \eqref{2.2}.
Next we define
\begin{equation}\label{5.8}
g(z+w):=\log f(z+w)=\log\sqrt{\pi}+\log\Gamma(\frac{z+w+1}{2})
+\log\Gamma(\frac{z+w+2}{2}),
\end{equation}
and by \eqref{5.4} we have for $k\geq 1$,
\begin{equation}\label{5.9}
g^{(k)}(z)
=2^{-k}\left(\psi^{(k-1)}(\tfrac{z+1}{2})-\psi^{(k-1)}(\tfrac{z+2}{2})\right).
\end{equation}
Expanding $g(z+w)$ as a function of $w$ in a neighborhood of $0\in D$, we get 
\[
g(z+w) = g(z)+\sum_{k=1}^\infty g^{(k)}(z)\frac{w^k}{k!},
\]
and with \eqref{5.1} and \eqref{5.8} we have
\begin{align}
f(z+w)&=\exp(g(z+w))
=\sqrt{\pi}\frac{\Gamma(\frac{z+1}{2})}{\Gamma(\frac{z+2}{2})}
\cdot\exp\left(\sum_{k=1}^\infty g^{(k)}(z)\frac{w^k}{k!}\right)\label{5.10}\\
&=\sqrt{\pi}\frac{\Gamma(\frac{z+1}{2})}{\Gamma(\frac{z+2}{2})}
\sum_{n=0}^\infty B_n(g^{(1)}(z),\ldots,g^{(n)}(z))\frac{w^n}{n!}.\nonumber
\end{align}
Taking the $p$th derivative of the right-most term in \eqref{5.10} and of the
series in \eqref{5.7}, where $p$ is a non-negative integer, we obtain
\begin{equation}\label{5.11}
\sum_{k=0}^\infty\frac{\binom{2k}{k}(-1)^p p!}{4^k(2k+1+z)^{p+1}}
=\frac{\sqrt{\pi}\Gamma(\frac{z+1}{2})}{2\Gamma(\frac{z+2}{2})}
\sum_{n=p}^\infty B_n(g^{(1)}(z),\ldots,g^{(n)}(z))\frac{w^{n-p}}{(n-p)!}.
\end{equation}
Finally, with $w=0$ in \eqref{5.11} we get \eqref{5.5}, which completes the
proof.
\end{proof}

Since $B_0=1$ and $B_1(g^{(1)}(z))=g^{(1)}(z)$, we immediately recover the
identities \eqref{2.2} and \eqref{4.1} from Theorem~\ref{thm:5.1}. Furthermore,
with Table~1 and \eqref{5.6} we see that
\[
B_2(g^{(1)}(z),g^{(2)}(z))=
\frac{1}{4}\left(\psi(\tfrac{z+1}{2})-\psi(\tfrac{z+2}{2})\right)^2
+\frac{1}{4}\left(\psi'(\tfrac{z+1}{2})-\psi'(\tfrac{z+2}{2})\right),
\]
and therefore \eqref{5.5} with $p=2$ gives \eqref{4.15}.

We now consider the special case where $z$ is a non-negative integer. As 
before, we need to distinguish between even and odd integers. For the following
result we require the generalized harmonic numbers defined for integers
$k\geq 1$ by
\begin{equation}\label{5.12}
H_0^{(k)}:=1\quad\hbox{and}\quad H_n^{(k)}:=\sum_{j=1}^n\frac{1}{j^k},\quad 
n=1, 2, 3,\ldots.
\end{equation}

\begin{lemma}\label{lem:5.2}
With $g^{(k)}(z)$ as defined in \eqref{5.6}, we have
\begin{align}
g^{(1)}(2m) &= H_{2m}-H_m-\log{2}\quad (m\geq 0),\label{5.13} \\
g^{(1)}(2m-1) &= -\left(H_{2m}-H_m-\log{2}\right)-\frac{1}{2m}\quad (m\geq 1),
\label{5.14}
\end{align}
and for $k\geq 2$,
\begin{align}
\frac{(-1)^kg^{(k)}(2m)}{(k-1)!}&=(1-2^{1-k})\zeta(k)
+2^{1-k}H_m^{(k)}-H_{2m}^{(k)},\label{5.15}\\
\frac{(-1)^kg^{(k)}(2m-1)}{(k-1)!}&=(2^{1-k}-1)\zeta(k)
-2^{1-k}H_m^{(k)}+H_{2m}^{(k)}+\frac{1}{(2m)^k},
\label{5.16}
\end{align}
valid for $m\geq 0$, resp.\ $m\geq 1$, where $\zeta(k)$ is the Riemann zeta 
function.
\end{lemma}

\begin{proof}
By \eqref{5.6} and \eqref{3.5} we have
\begin{align*}
g^{(1)}(2m)&=\frac{1}{2}\left(\psi(m+\tfrac{1}{2})-\psi(m+1)\right)
=\frac{1}{2}\left(2H_{2m}-2H_m-2\log{2}\right),\\
g^{(1)}(2m-1)&=\frac{1}{2}\left(\psi(m)-\psi(m+\tfrac{1}{2})\right)
=\frac{1}{2}\left(H_{m-1}+H_m-2H_{2m}+2\log{2}\right),
\end{align*}
which give \eqref{5.13} and \eqref{5.14}, respectively.

For $k\geq 2$, we first note that by \cite[eq.~(25.11.12)]{DLMF} we have
\begin{equation}\label{5.17}
\psi^{(k-1)}(z)=(-1)^k(k-1)!\zeta(k,z),
\end{equation}
where $\zeta(s,a)$ is the Hurwitz zeta function with complex parameter $a$,
defined by
\begin{equation}\label{5.18}
\zeta(s,a) = \sum_{n=0}^\infty\frac{1}{(n+a)^s},\quad \rm{Re}(s)>1, 
a\neq 0, -1, -2,\cdots
\end{equation}
With \eqref{5.17} we can now rewrite \eqref{5.6} as
\begin{equation}\label{5.19}
g^{(k)}(z)=\left(\tfrac{-1}{2}\right)^k(k-1)!
\left(\zeta(k,\tfrac{z+1}{2})-\zeta(k,\tfrac{z+2}{2})\right)\quad (k\geq 2).
\end{equation}
Using the definitions \eqref{5.18} and \eqref{5.12}, we see that for integers
$m\geq 0$,
\begin{equation}\label{5.20}
\zeta(k,m+1) = \zeta(k)-H_m^{(k)}.
\end{equation}
Similarly,
\begin{equation}\label{5.21}
\zeta(k,m+\tfrac{1}{2}) 
= \zeta(k,\tfrac{1}{2})-\sum_{n=1}^m\frac{1}{(n-\tfrac{1}{2})^k}
= (2^k-1)\zeta(k)-\sum_{n=1}^m\frac{1}{(n-\tfrac{1}{2})^k},
\end{equation}
where we have used the identity (25.11.11) in \cite{DLMF}. The sum in 
\eqref{5.21} can be expressed in terms of harmonic numbers:
\begin{align}
\sum_{n=1}^m\frac{1}{(n-\tfrac{1}{2})^k}&=2^k\sum_{n=1}^m\frac{1}{(2n-1)^k}
=2^k\left(\sum_{n=1}^{2m}\frac{1}{n^k}-\sum_{n=1}^m\frac{1}{(2n)^k}\right)
\label{5.22}\\
&= 2^kH_{2m}^{(k)}-H_m^{(k)}.\nonumber
\end{align}
Combining this with \eqref{5.21} and \eqref{5.20}, we get
\begin{equation}\label{5.23}
\zeta(k,m+\tfrac{1}{2})-\zeta(k,m+1)
= (2^k-2)\zeta(k) + 2H_m^{(k)}-2^kH_{2m}^{(k)}.
\end{equation}
Finally, combining \eqref{5.23} with \eqref{5.19} for $z=2m$, we immediately
get \eqref{5.15}. The identity \eqref{5.16} is similarly obtained by 
combining \eqref{5.19}--\eqref{5.22}, this time for $z=2m-1$. The proof is now
complete.
\end{proof}

For further use, we evaluate the following small cases.

\begin{corollary}\label{cor:5.3}
With $g^{(k)}(z)$ as defined in \eqref{5.6}, we have for integers $k\geq 2$,
\begin{align}
g^{(1)}(0)=-\log{2},\quad
&g^{(k)}(0)=(-1)^k(k-1)!\left(1-2^{1-k}\right)\zeta(k),\label{5.24} \\
g^{(1)}(1)=-1+\log{2},\quad
&g^{(k)}(1)=(-1)^k(k-1)!\left(1-\left(1-2^{1-k}\right)\zeta(k)\right).\label{5.25}
\end{align}
\end{corollary}

As is well known, Euler's formula allows us to write $\zeta(k)$ as a rational
multiple of $\pi^k$ when $k$ is even; see, e.g., \cite[eq.~25.6.2]{DLMF}. In
particular,
\begin{equation}\label{5.26}
\zeta(2)=\frac{\pi^2}{6},\quad \zeta(4)=\frac{\pi^4}{90},\quad
\zeta(6)=\frac{\pi^6}{945}.
\end{equation}

To obtain some specific series evaluations, we first consider \eqref{5.5}
with $z=0$:
\begin{equation}\label{5.27}
\sum_{k=0}^\infty\frac{\binom{2k}{k}}{4^k(2k+1)^{p+1}}
=\frac{(-1)^p\pi}{2p!}\cdot B_{p}(g^{(1)}(0),\ldots,g^{(p)}(0)).
\end{equation}
Then with $p=0,1,2$ and using \eqref{5.24}, we recover the identities
\eqref{1.1}, \eqref{4.4}, and \eqref{4.16}, respectively. Similarly, the cases
$p=3$ and $p=4$ lead to the following identities.

\begin{corollary}\label{cor:5.4}
\begin{align}
\sum_{k=0}^\infty\frac{\binom{2k}{k}}{4^k(2k+1)^4}
&=\frac{\pi\log^3{2}}{12}+\frac{\pi^3\log{2}}{48}+\frac{\pi\zeta(3)}{8},\label{5.28}\\
\sum_{k=0}^\infty\frac{\binom{2k}{k}}{4^k(2k+1)^5}
&=\frac{\pi\log^4{2}}{48}+\frac{\pi^3\log^2{2}}{96}+\frac{\pi\zeta(3)\log{2}}{8}
+\frac{19\pi^5}{11520}.\label{5.29}
\end{align}
\end{corollary}

Next, using the fact that $\Gamma(\frac{3}{2})=\frac{1}{2}\sqrt{\pi}$, we get
from \eqref{5.5} with $z=1$,
\begin{equation}\label{5.30}
\sum_{k=0}^\infty\frac{\binom{2k}{k}}{4^k(2k+2)^{p+1}}
=\frac{(-1)^p}{p!}\cdot B_{p}(g^{(1)}(1),\ldots,g^{(p)}(1)).
\end{equation}
Once again we note that the cases $p=0,1,2$, this time with \eqref{5.25}, 
give the identities \eqref{1.1}, \eqref{4.5}, and \eqref{4.17}, respectively.
With $p=3$ and $p=4$ we get the following, leaving $\zeta(k)$ for even
$k$ unchanged this time.

\begin{corollary}\label{cor:5.5}
\begin{align}
\sum_{k=0}^\infty\frac{\binom{2k}{k}}{4^k(2k+2)^4}
&=\left(1-\log{2}+\frac{\log^2{2}}{2}-\frac{\log^3{2}}{6}\right)
-\frac{\zeta(2)}{4}(1-\log{2})-\frac{\zeta(3)}{4},\label{5.31}\\
\sum_{k=0}^\infty\frac{\binom{2k}{k}}{4^k(2k+2)^5}
&=\left(1-\log{2}+\cdots+\frac{\log^4{2}}{4!}\right)
-\frac{\zeta(2)}{4}\left(1-\log{2}+\frac{\log^2{2}}{2}\right)\label{5.32}\\
&\qquad-\frac{\zeta(3)}{4}(1-\log{2})-\frac{7\zeta(4)}{32}+\frac{\zeta(2)^2}{32}.\nonumber
\end{align}
\end{corollary}

\noindent
{\bf Remark.} The occurrence of the partial sums of the series
$\sum(-\log{2})^n/n!$ can be explained as follows. Among the partial derivatives
of the Bell polynomials we have
\begin{equation}\label{5.33}
\frac{\partial}{\partial x_1}B_p(x_1,\ldots x_{p-1},x_p) 
= p\cdot B_{p-1}(x_1,\ldots x_{p-1});
\end{equation}
see \cite[p.~266]{Be}. If we replace $\log{2}$ on the right-hand sides of 
\eqref{5.30}--\eqref{5.32} by $x$, recalling that $g^{(1)}(1)=-1+\log{2}$, and
then differentiate with respect to $x$, we get \eqref{5.31} from \eqref{5.32}.

Going in the opposite direction, we can use this derivative argument and the 
fact that $B_{p}(g^{(1)}(1),\ldots,g^{(p)}(1))$ always contains the single 
summand $g^{(p)}(1)$, which is a rational multiple of $\zeta(p)$. In this way 
we can show that the part of the right-hand side of \eqref{5.30} that contains 
no products of zeta values is of the form
\[
L(p)-\sum_{j=2}^p\frac{j-1}{2^j}L(p-j)\zeta(j),\quad\hbox{where}\quad
L(n):=\sum_{i=0}^n\frac{(-\log{2})^i}{i!}.
\]
However, this does not give a general formula for \eqref{5.30}, as can 
already be seen from the right-most term in \eqref{5.32}.

A similar argument can be used in the easier case of \eqref{5.27}, where the 
part that contains no products of zeta values is of the form
\[
\frac{\pi}{2}\cdot\frac{(\log{2})^p}{p!}
+\frac{\pi}{8}\sum_{j=2}^p\frac{(\log{2})^{p-j}}{(p-j)!}\zeta(j).
\]
This is consistent with \eqref{4.4}, \eqref{4.16}, \eqref{5.28}, and
\eqref{5.29}.

\section{Integral representations}\label{sec:6}

As part of the proof of Theorem~\ref{thm:2.1} we have already used an integral
representation in \eqref{2.3} for a class of infinite series. In this section
we will extend this representation. Much of the content of this section is not
new, but we include it for the sake of completeness and for a different 
perspective.

Our point of departure is a remarkable identity due to Boyadzhiev
\cite[eq.~(3.36)]{Bo}, namely
\begin{equation}\label{6.1}
\sum_{n=0}^\infty\binom{2n}{n}a_n\frac{z^n}{4^n}
=\frac{1}{\sqrt{z+1}}\sum_{n=0}^\infty\left(\frac{z}{z+1}\right)^n
\frac{\binom{2n}{n}}{4^n}\sum_{k=0}^n\binom{n}{k}a_k,
\end{equation}
valid for all complex $z$ with $|z|<1$, and where $a_0, a_1,\ldots$ is a 
bounded sequence. Note the occurrence of the terms $\binom{2n}{n}/4^n$ on both
sides. Also, the right-most sum in \eqref{6.1} is the binomial transform of the
sequence $(a_k)$.

We are now ready to state and prove the following result.

\begin{theorem}\label{thm:6.1}
For any integer $p\geq 0$ we have
\begin{equation}\label{6.2}
\sum_{k=0}^\infty\frac{\binom{2k}{k}}{4^k(2k+1)^{p+1}}z^k
=\frac{(-1)^p}{p!}\int_0^1\frac{\log^p{x}}{\sqrt{1-zx^2}}dx.
\end{equation}
\end{theorem}

\begin{proof}
We substitute $a_k=(2k+1)^{-p}$ into \eqref{6.1} and use the integral 
representation 
\begin{equation}\label{6.3}
\sum_{k=0}^n\frac{\binom{n}{k}}{(2k+1)^{p+1}}
=\frac{(-1)^p}{p!}\int_0^1(1+x^2)^n\log^p{x}\,dx,
\end{equation}
which is a special case of the identity (4.272.15) in \cite[p.~552]{GR}.
We then get
\begin{align}
\sum_{n=0}^\infty\binom{2n}{n}&\frac{z^n}{4^n(2n+1)^{p+1}} \label{6.4}\\
&=\frac{1}{\sqrt{z+1}}\sum_{n=0}^\infty\left(\frac{z}{z+1}\right)^n
\frac{\binom{2n}{n}}{4^n}\frac{(-1)^p}{p!}\int_0^1\log^p{x}(1+x^2)^ndx\nonumber\\
&=\frac{1}{\sqrt{z+1}}\frac{(-1)^p}{p!}\int_0^1\log^p{x}
\left(\sum_{n=0}^\infty\left(\frac{z}{z+1}\right)^n\frac{\binom{2n}{n}}{4^n}
(1+x^2)^n\right)dx.\nonumber
\end{align}
Since we are dealing with absolutely convergent power series, the interchange
of integral and series is legitimate. The sum in the last integral can be 
evaluated by \eqref{1.1a}, giving
\[
\sum_{n=0}^\infty\binom{2n}{n}\left(\frac{z(1+x^2)}{4(z+1)}\right)^n
=\frac{1}{\sqrt{1-\frac{z}{z+1}(1+x^2)}}.
\]
This, together with \eqref{6.4}, immediately gives \eqref{6.2} after replacing
$n$ by $k$.
\end{proof}

\noindent
{\bf Example.} Let $p=0$ in \eqref{6.2}. If we set $z=y^2$, then
\[
\sum_{k=0}^\infty\frac{\binom{2k}{k}y^{2k}}{4^k(2k+1)}
=\int_0^1\frac{dx}{\sqrt{1-(yx)^2}}
=\frac{1}{y}\int_0^y\frac{du}{\sqrt{1-u^2}} = \frac{\arcsin{y}}{y},
\]
which is a special case of identity (16) in \cite{DV}. In particular, with 
$y=1$ we recover the first identity in \eqref{1.1}.

\medskip
\noindent
{\bf Remarks.}
(1) The identity \eqref{6.2} is not new. In fact, it is a special case of the
integral (4.272.14) in \cite[p.~552]{GR} with $p=1$, $q=2$, $s=1/2$, $r=p+1$, 
and with $x$ replaced by $x\sqrt{-1}$.

(2) On the other hand, if we combine \eqref{6.2} with \eqref{5.27}, then we get
the following apparently new integral evaluation in terms of Bell polynomials.

\begin{corollary}\label{cor:6.2}
For any integer $p\geq 0$ we have
\begin{equation}\label{6.5}
\int_0^1\frac{\log^p{x}}{\sqrt{1-x^2}}dx
= \frac{\pi}{2}\cdot B_p(c_1,\ldots, c_p),
\end{equation}
where $c_1=-\log{2}$ and $c_k=(-1)^k(k-1)!(1-2^{1-k})\zeta(k)$ for $k\geq 2$.
\end{corollary}

We now generalize Theorem~\ref{thm:6.1} and its proof by introducing a new
variable. The results are then given in terms of Euler's beta function
$B(a,b)$, defined by \eqref{2.3a}.

\begin{theorem}\label{thm:6.3}
For any $w\in{\mathbb C}\setminus\{0,-2,-4,\ldots\}$ and any integer $p\geq 0$ 
we have
\begin{equation}\label{6.6}
\sum_{k=0}^\infty\frac{\binom{2k}{k}}{4^k(2k+w)^{p+1}}
=\frac{(-1)^p}{2p!}\cdot\frac{d^p}{dw^p}B\big(\frac{1}{2},\frac{w}{2}\big),
\end{equation}
and the generating function
\begin{equation}\label{6.7}
\sum_{p=0}^\infty
\left(\sum_{k=0}^\infty\frac{\binom{2k}{k}}{4^k(2k+w)^{p+1}}\right)z^p
=\frac{1}{2}\cdot B\big(\frac{1}{2},\frac{w-z}{2}\big).
\end{equation}
\end{theorem}

\begin{proof}
With the aim of applying \eqref{6.1} with $a_k=(2k+w)^{-p}$, we use the 
integral representation
\begin{equation}\label{6.8}
\frac{1}{(2k+w)^{p+1}}
=\frac{(-1)^p}{p!}\int_0^1 x^{2k+w-1}\log^p{x}\,dx,
\end{equation}
which can be found in \cite[p.~551]{GR} as (4.272.6). With a binomial
expansion, \eqref{6.8} then gives
\[
\sum_{k=0}^n\frac{\binom{n}{k}}{(2k+w)^{p+1}}
=\frac{(-1)^p}{p!}\int_0^1 x^{w-1}(1+x^2)^n\log^p{x}\,dx,
\]
which is an extension of \eqref{6.3}. In exactly the same way as in the proof
of Theorem~\ref{thm:6.1}, we then get
\begin{equation}\label{6.9}
\sum_{n=0}^\infty\frac{\binom{2n}{n}}{4^n(2n+w)^{p+1}}z^n
=\frac{(-1)^p}{p!}\int_0^1\frac{x^{w-1}\log^p{x}}{\sqrt{1-zx^2}}dx.
\end{equation}
This obviously extends \eqref{6.2}. The integral on the right, for $z=1$, can 
be found in \cite{PrE} as (2.6.5.1) with $\alpha=w$, $\beta=1/2$, and $\mu=2$, 
namely
\[
\int_0^1\frac{x^{w-1}\log^p{x}}{\sqrt{1-x^2}}dx
=\frac{1}{2}\cdot\frac{d^p}{dw^p}B\big(\frac{1}{2},\frac{w}{2}\big).
\]
This, with \eqref{6.9}, gives \eqref{6.6}. The identity \eqref{6.7} is then
obvious as a Maclaurin expansion in $z$ of the right-hand side of \eqref{6.7}.
\end{proof}

\noindent
{\bf Remarks.} (1) When $p=0$, the identity \eqref{6.6} reduces to \eqref{2.2}
with $w=z+1$ and with the definition \eqref{2.3a}, noting that 
$\Gamma(\frac{1}{2})=\sqrt{\pi}$.

(2) At first sight, the identities \eqref{5.5} and \eqref{6.6} appear to be
entirely different expressions for the same series. However, this becomes 
less surprising if we realize that higher derivatives of composite functions,
as we have in \eqref{6.6}, are closely related to Bell polynomials through
the Fa\'a di Bruno formula; see, e.g., \cite[Sect.~3.4]{Co}.

\section{Further Remarks}\label{sec:7}

{\bf 1.} An alternative approach to Theorem~\ref{thm:6.3} would involve going in
the opposite direction to the proof given above. As mentioned in the previous
Remark (1), Theorem~\ref{thm:2.1} is equivalent to 
\begin{equation}\label{6.10}
\sum_{k=0}^\infty\frac{\binom{2k}{k}}{4^k(2k+1+z)}
=\frac{1}{2}\cdot B\big(\frac{1}{2},\frac{z+1}{2}\big).
\end{equation}
If we replace $z$ by $w-z-1$ and expand
\[
\frac{1}{2k+w-z}=\frac{1}{2k+w}\cdot\frac{1}{1-\frac{z}{2k+w}}
=\frac{1}{2k+w}\sum_{p=0}^\infty\frac{z^p}{(2k+w)^p},
\]
then with \eqref{6.10} we immediately get \eqref{6.7}. The identity \eqref{6.6}
then follows from the uniqueness of Taylor expansions.

However, the point of the original proof was to display some interesting
integral representations and show their usefulness in dealing with infinite 
series such as the ones in this paper.

\medskip
{\bf 2.} Yet another approach to Theorem~\ref{thm:6.3} that does not use 
integral representations is to begin with the series
\begin{equation}\label{6.11}
B(x,y)=\sum_{k=0}^\infty\binom{k-x}{k}\frac{1}{y+k},
\end{equation}
which can be found, for instance, in an equivalent form as identity (8.382.1)
in \cite[p.~909]{GR}. This gives
\begin{equation}\label{6.12}
\frac{1}{2}B\big(\frac{1}{2},\frac{w}{2}\big)
=\sum_{k=0}^\infty\binom{k-\frac{1}{2}}{k}\frac{1}{w+2k}.
\end{equation}
Using the duplication formula \eqref{2.1} with $z=k+\frac{1}{2}$ in the form
\begin{equation}\label{6.13}
\Gamma(k+\tfrac{1}{2})=\frac{\Gamma(2k+1)\sqrt{\pi}}{\Gamma(k+1)4^k}
\end{equation}
and recalling that $\Gamma(\frac{1}{2})=\sqrt{\pi}$, we get
\begin{equation}\label{6.14}
\binom{k-\frac{1}{2}}{k}
=\frac{\Gamma(k+\frac{1}{2})}{\Gamma(\frac{1}{2})\Gamma(k+1)}
=\frac{1}{4^k}\binom{2k}{k}.
\end{equation}
With \eqref{6.12}, this gives \eqref{6.10} and therefore, according to 
Subsection 1 above, it also leads to the desired identity \eqref{6.6}.

\medskip
{\bf 3.} We can also use the identity \eqref{6.11} to obtain variants to many
of the results in this paper. For one such class of variants, we set 
$x=\nu+\frac{1}{2}$ in \eqref{6.11}, where $\nu\geq 0$ is an integer. Then by
iterating the basic recurrence $\Gamma(z+1)=z\Gamma(z)$ and using \eqref{6.14},
we get
\begin{align}
\sum_{k=0}^\infty\frac{\binom{2k}{k}}{4^k(k+z)}
&\cdot\frac{1}{(2k-1)(2k-3)\cdots(2k-2\nu+1)} \label{6.15} \\
&\qquad =(-1)^\nu\frac{2^\nu\nu!}{(2\nu)!}
\cdot\frac{\Gamma(\nu+\frac{1}{2})\Gamma(z)}{\Gamma(\nu+\frac{1}{2}+z)},\nonumber
\end{align}
valid for all $z$ for which the gamma terms are defined. 

First, setting $z=m$ (respectively $z=m+\frac{1}{2}$) and using \eqref{6.13},
we get the following two classes of identities.

\begin{corollary}\label{cor:6.4}
For all integers $\nu\geq 0$ and $m\geq 1$, we have
\begin{align}
\sum_{k=0}^\infty\frac{\binom{2k}{k}}{4^k(2k+2m)}
&\cdot\frac{1}{(2k-1)(2k-3)\cdots(2k-2\nu+1)}\label{6.16}\\
&\qquad=(-1)^\nu 2^{2m+\nu-1}\frac{(m-1)!(m+\nu)!}{(2m+2\nu)!},\nonumber
\end{align}
and for $m\geq 0$,
\begin{align}
\sum_{k=0}^\infty\frac{\binom{2k}{k}}{4^k(2k+2m+1)}
&\cdot\frac{1}{(2k-1)(2k-3)\cdots(2k-2\nu+1)}\label{6.17}\\
&\qquad=\frac{(-1)^\nu}{2^{2m+\nu+1}}\frac{(2m)!\pi}{(m+\nu)!m!}.\nonumber
\end{align}
\end{corollary}

{\bf 4.} We can obtain further identities by differentiating \eqref{6.15}. 
In analogy to Corollary~\ref{cor:4.1}, we then get from \eqref{6.15},
\begin{align}
\sum_{k=0}^\infty\frac{\binom{2k}{k}}{4^k(k+z)^2}
&\cdot\frac{1}{(2k-1)(2k-3)\cdots(2k-2\nu+1)}\label{6.18}\\
&=(-1)^\nu\frac{2^\nu\nu!}{(2\nu)!}
\cdot\frac{\Gamma(\nu+\frac{1}{2})\Gamma(z)}{\Gamma(\nu+\frac{1}{2}+z)}
\left(\psi(\nu+\tfrac{1}{2}+z)-\psi(z)\right).\nonumber
\end{align}
Using the evaluations \eqref{6.13} and \eqref{3.5} for the digamma and gamma
functions, respectively, one could easily state general identities, analogous
to Corollary~\ref{cor:6.4}. Instead, we just give two examples for $z=1$ and
$\nu=1$, resp.\ $\nu=2$:
\begin{align*}
\sum_{k=0}^\infty\frac{\binom{2k}{k}}{4^k(k+1)^2(2k-1)}
&=\frac{4\log{2}}{3}-\frac{4}{9},\\
\sum_{k=0}^\infty\frac{\binom{2k}{k}}{4^k(k+1)^2(2k-1)(2k-3)}
&=\frac{92}{225}-\frac{4\log{2}}{15}.
\end{align*}

{\bf 5.} As a final application of this method, we set $x=-\frac{1}{2}$ in
\eqref{6.11} and note that 
\[
\binom{k+\frac{1}{2}}{k} = (2k+1)\binom{k-\frac{1}{2}}{k}
=\frac{2k+1}{4^k}\binom{2k}{k}.
\]
Then upon differentiation, as before, we get
\[
\sum_{k=0}^\infty\frac{\binom{2k}{k}(2k+1)}{4^k(k+z)^2}
=\frac{\Gamma(-\frac{1}{2})\Gamma(z)}{\Gamma(z-\frac{1}{2})}
\left(\psi(z-\tfrac{1}{2})-\psi(z)\right).
\]
For $z=1$, we use \eqref{3.5} and the fact that 
$\Gamma(-\frac{1}{2})=-2\Gamma(\frac{1}{2})$ to get as a specific example,
\[
\sum_{k=0}^\infty\frac{\binom{2k}{k}(2k+1)}{4^k(2k+2)^2} = \log{2}.
\]
It will be clear that numerous other identities and classes of identities can
be obtained in a similar way.

\end{document}